\documentclass[twoside]{amsart}
\usepackage{amssymb,amsmath}
\usepackage{verbatim}
\usepackage[dvips]{graphicx}
\usepackage{color}

\newcommand{\Aut}[1]{\mathrm{Aut}\left(#1\right)}     %  Aut(G)
\newcommand{\Inn}[1]{\mathrm{Inn}\left(#1\right)}     %  Inn(G)
\newcommand{\Out}[1]{\mathrm{Out}\left(#1\right)}     %  Out(G)
\newcommand{\lto}{\longrightarrow}

\newcommand{\dm}{\mathrm{dim}}                 %   dimensione
\newcommand{\rk}{\mathrm{rk}}                  %   rango
\newtheorem{thm}{Theorem}%[section]
\newtheorem{pro}{Proposition}%[section]
\newtheorem{lem}[pro]{Lemma}%[section]
\begin{document}
\title[The Solvability of Groups with nilpotent minimal coverings]{The Solvability 
of Groups\\ with nilpotent minimal coverings}

\author{Russell D. Blyth}
   \address{Department of Mathematics and Computer Science, Saint Louis
            University\\220 N. Grand Blvd.\\St. Louis, MO 63103, U.S.A.}
   \email[R.~D.~Blyth]{blythrd@slu.edu}
\author{Francesco Fumagalli}
   \address{Dipartimento di Matematica e Informatica ``Ulisse Dini'', Universit\`a di
            Firenze\\ Viale Morgagni, 67/A, 50134 Firenze, Italy}
   \email[F.~Fumagalli]{fumagalli@math.unifi.it}    
\author{Marta Morigi}
   \address{Dipartimento di Matematica, Universit\`a di Bologna\\ Piazza
            di Porta San Donato 5, 40126 Bologna, Italy}
   \email[M.~Morigi]{marta.morigi@unibo.it}

\maketitle
\begin{center} {\it In memory of L\'{a}szl\'{o} Kov\'{a}cs}
\end{center}

\begin{abstract}
A \emph{covering} of a group is a finite set of proper subgroups whose union 
is the whole group. A covering is \emph{minimal} if there is no covering of 
smaller cardinality, and it is \emph{nilpotent} if all its members are nilpotent 
subgroups. We complete a proof that every group that has a nilpotent minimal covering 
is solvable, starting from the previously known result that a minimal counterexample 
is an almost simple finite group.
\end{abstract}
\vspace{0.5cm}



\section{Introduction}

A \emph{covering} (or \emph{cover}) for a group $G$ is a finite collection of proper 
subgroups whose union is all of $G$. A \emph{minimal covering} for $G$ is a covering 
which has minimal cardinality among all the coverings of $G$. The size of a minimal 
covering of a group $G$ is denoted $\sigma(G)$  and is called the 
\emph{covering number} of $G$. 
Since the first half of the last century a lot of attention has been given to 
determining which numbers can occur as covering numbers for groups, and, when 
possible, to characterize groups having the same value of $\sigma(G)$. The earlier 
works date back to G.~Scorza (\cite{Scorza1926}) and D.~Greco (\cite{Greco1951}, 
\cite{Greco1953}, \cite{Greco1956}). The terminology ``minimal covering'' appears in 
the celebrated paper \cite{Tom1997} of M.~J.~Tomkinson. Also worth mentioning are 
\cite{HR1959}, \cite{BBM1970} and \cite{Cohn1994}. 
More recent works determine bounds, and also exact values of $\sigma(G)$, for various 
classes of finite groups (see for instance \cite{BFS1999}, \cite{Lucido2003}, 
\cite{Maroti2005}, \cite{Holmes2006} and \cite{HolmesMaroti2010}). 

Here we are interested in minimal coverings of groups by proper subgroups 
with restricted properties. For example, in \cite[Theorem 2]{BS2001} R.~Bryce 
and L.~Serena show that a group that has a minimal covering consisting of 
abelian subgroups is solvable of very restricted structure. In \cite{BS2008} 
the same authors treat the case of groups that admit a minimal covering with 
all members nilpotent, that is, a {\sl nilpotent minimal covering}. 
They state the following: \\

\noindent {\bf Conjecture.}\label{BS_Conj} {\it Only solvable groups can admit a 
nilpotent minimal covering.}\\

Their main result (\cite[Proposition 2.1]{BS2008}) is  a
reduction to the almost simple case, namely if there is an insolvable group with a 
nilpotent minimal covering then there is a finite almost simple such group. Bryce and 
Serena also show that several classes of finite almost simple groups (among them the 
alternating and symmetric groups, the projective special/general linear groups, the 
Suzuki groups, and the 26 sporadic groups) do not have nilpotent minimal coverings.

Our main result is the following. 

\begin{thm}\label{goal} No finite almost simple group has a nilpotent minimal 
covering.
\end{thm}
As an immediate corollary we complete the proof of the aforementioned conjecture.
\begin{thm}\label{goalcor} Every group that has a nilpotent minimal covering is 
solvable.
\end{thm}
The structure of solvable groups with such a minimal covering is well 
understood and  can be found in \cite[Theorem 11]{BS2001}.

A reasonable indication of the truth of Theorem \ref{goal} is suggested by the 
fact that in a finite non abelian simple group the order of the largest nilpotent 
subgroups is always much smaller than the order of the group (see \cite{Vdovin2000}).  

Our proof of Theorem \ref{goal} makes use of the classification of finite non abelian 
simple groups, and it can be outlined as follows. We start by taking a minimal order 
counterexample $G$, which is therefore  an almost simple group, say $S \le G \le 
\Aut{S}$, where $S$ is a non abelian simple group. 
If $S$ is a group of Lie type we reduce to the
cases when $S$ has Lie rank one or twisted Lie rank one, or $S$ has Lie rank 
two and $G$ contains a graph, or a graph-field, automorphism 
of $S$. Then we reduce to $G/S$ cyclic and we use a technical lemma (Lemma 
\ref{centralizers}) 
to eliminate the possibility that $G$ is itself not simple (Proposition 
\ref{reduction_to_simple}). Finally we prove that no finite simple 
group can be a counterexample (Proposition \ref{end_simple}).\\

Recently nilpotent coverings and their connections with maximal
non-nilpotent subsets in finite simple groups of Lie type have been studied in 
\cite{ABG2013}. \\

The notation of this paper is standard and mostly follows the book \cite{GLS3}. 
We remark that for the classical groups we have preferred to use the 
`classical' notation rather than Artin's single letter notation. Therefore we use 
$PSL(n,q)$ instead of $A_{n-1}(q)$  or $L_n(q)$, and similarly $PSp(2n,q)$ 
for $B_{2n}(q)$ and $PSU(n,q)$ for $\null^2 A_{n-1}(q)$ or $U_n(q)$. Note also that 
whenever we write $PSU(n,q)$ we mean that this group is defined over the field of 
order $q^2=p^f$ ($p$ a prime). Differently from \cite{GLS3}, we denote the Suzuki and 
the Ree groups over the fields $\mathbb{F}_{2^f}$ and $\mathbb{F}_{3^f}$ ($f$ odd), 
by $\null^2 B_2(2^f)$, $\null^2 G_2(3^f)$, instead of $\null^2 B_2(2^{\frac{f}{2}})$ 
and $\null^2 G_2(3^{\frac{f}{2}})$.\\

\section{Proofs of Theorems \ref{goal} and \ref{goalcor}}
We start with a simple but important observation. Assume that 
${\mathcal A} = \{A_1,\ldots, A_\sigma\}$
is a minimal covering of a group $G$, that is,
$$G=A_1\cup A_2\cup \ldots \cup A_{\sigma}$$
and $G$ is not the setwise union of fewer than $\sigma$ proper subgroups. 
Then for every 
$1\le i < j \le \sigma$, $\langle A_i,A_j \rangle = G$, since otherwise 
we could replace the subgroups $A_i$ and $A_j$ in ${\mathcal A}$ with 
$\langle A_i,A_j \rangle $, obtaining a covering of $G$ with fewer than $\sigma $ 
members. We will use this simple fact often.\\

The proof of Theorem \ref{goal} depends on understanding 
the structure of the finite simple groups of Lie type and the corresponding 
simple linear algebraic groups. Lemma \ref{prop_uno} is a key step in our proof. 
We first recall some important facts regarding algebraic groups.

A {\it regular unipotent element} of an algebraic group $\overline{G}$ is an element 
$g$ of $\overline{G}$ such that $\dm (C_{\overline{G}}(g))=\rk (\overline{G})$. 
The following result can be found in \cite{Carter85}, 
Propositions 5.1.2 and 5.1.3.
\begin{pro}\label{r.u.e.}
Let $\overline{G}$ be a connected reductive group. Then $\overline{G}$ admits regular 
unipotent elements and these elements form a unique conjugacy class in 
$\overline{G}$. 
Moreover, if $u$ is a unipotent element of $\overline{G}$, then the following 
conditions on $u$ are equivalent:
\begin{enumerate}
\item[(a)] $u$ is regular,
\item[(b)] $u$ lies in a unique Borel subgroup of $\overline{G}$, and
\item[(c)]  $u$  is conjugate to an element of the form 
$\prod_{\alpha\in \Phi^+} x_{\alpha}(\lambda_{\alpha})$ with 
$\lambda_{\alpha}\ne 0$ for all $\alpha\in \Delta$ (where
$\Phi^+$ and $\Delta$ denote, respectively, a positive system of roots and 
its fundamental system).
\end{enumerate}
\end{pro}
In particular, if $\overline{G}$ is a simple linear algebraic group over an 
algebraically 
closed field in characteristic $p$ and $F$ is a Frobenius endomorphism of 
$\overline{G}$, 
then the finite group of Lie type $G=\overline{G}^F$ contains $p$-elements 
(that we still call \emph{regular unipotent}) which have the property that 
each lies in a unique Sylow $p$-subgroup of $G$ (see 
\cite[Proposition 5.1.7]{Carter85}).

\begin{lem}\label{prop_uno} Let $G$ be an almost simple group whose socle $S$ is a 
group of Lie type in characteristic $p$. 
Suppose that $\mathcal{A}=\{A_i\}_{i=1}^{\sigma}$ is a nilpotent minimal covering 
of $G$. Then the following hold:
\begin{enumerate}
\item[(a)] $\sigma$ is greater than the number $n_p(S)$ of Sylow $p$-subgroups 
of $S$, and 
\item[(b)] if $U$ is a Sylow $p$-subgroup of $S$, then $N_G(U)$ is a maximal subgroup 
of $G$.
\end{enumerate}
\end{lem}
\begin{proof} 
Let $u$ be a regular unipotent element of $S$ and let $U$ be the unique 
Sylow $p$-subgroup of $S$ containing $u$. Assume that $u\in A_i$. 
Since $A_i$ is nilpotent, $O_{p'}(A_i)\leq C_G(u)$; in particular $O_{p'}(A_i)$ 
normalizes $U$. 
%For let $g\in C_G(u)$ then $u=u^g\in U^g=U$, being
%$U$ the unique Sylow $p$-subgroup of $S$ containing $u$. 
We now prove that $O_p(A_i)\leq N_G(U)$. Let $y\in O_p(A_i)$. Then $y$ normalizes 
$O_p(A_i)\cap S$, and since $u\in O_p(A_i)\cap S=(O_p(A_i)\cap S)^y\leq U^y$ and $U$ 
is the unique Sylow $p$-subgroup of $S$ containing $u$, we have $U^y=U$, as we 
wanted. It follows that $A_i\leq N_G(U)$.
As $\langle A_i,A_j\rangle =G$ for $i\ne j$, two different members of $\mathcal{A}$ 
cannot normalize the same Sylow $p$-subgroup of $S$. This shows that 
$\sigma\geq \vert Syl_p(S)\vert $. Moreover, since a finite group is never the union 
of conjugates of a unique proper subgroup (\cite[Theorem 1]{Kantor}),
in fact $\sigma> \vert Syl_p(S)\vert $.

Assume now that $N_G(U)<K<G$. Then for every $k\in K\setminus N_G(U)$, the element 
$u^k$ is still regular unipotent in $S$ and lies in $U^k\ne U$. If $u\in A_i$ and 
$u^k\in A_j$, we have that for $i\ne j$, the subgroup  $A_j$ is contained in 
$N_G(U)^k$ and 
$$G=\langle A_i,A_j\rangle \leq \langle N_G(U), N_G(U)^k\rangle \leq K,$$ 
which is a contradiction. Thus $N_G(U)$ is maximal in $G$.
\end{proof}

We next determine in which of these groups $G$ the normalizer of a Sylow $p$-group of 
$S$ is a maximal subgroup of $G$.
\begin{pro}\label{List A}
Let $G$ be a finite almost simple group whose socle $S$ is a group of Lie type 
in characteristic $p$. Let $U$ be a Sylow $p$-subgroup of $S$. Then $N_G(U)$ is \
maximal in $G$ if and only if one of the following holds: 
\begin{enumerate}
\item[(a)] $S\in\{PSL(2,q), PSU(3,q), \null^2 B_2(q), \null^2 G_2(q) \} $, or 
\item[(b)] $S\in\{PSL(3,q), PSp(4,2^f), G_2(3^f)\}$ and $G$ contains a graph or 
graph-field 
automorphism of $S$.
\end{enumerate}
\end{pro}
\begin{proof}
Assume first that $G=S$ is simple. Then $B=N_S(U)$ is a Borel subgroup of $S$ and, 
by general BN-pair theory (\cite[Proposition 8.2.1 and Theorem 13.5.4]{Carter72}), 
the lattice of overgroups of $B$ in $S$ 
consists of $B$, the parabolic subgroups of $S$, and $S$. In particular, $N_S(U)$ 
is maximal in $S$ if and only if it is the unique parabolic subgroup of $S$, which is 
the case exactly when $S$ is of Lie rank one or, respectively, of twisted Lie rank 
one. Only the finite simple groups listed in (a) have this property.

Assume now that $G>S$ and let $S^*$ be the extension of $S$ by the diagonal and field 
automorphisms of $S$. The group $S^*$ has a BN-pair whose Borel subgroup is 
$B^*=N_{S^*}(U)$, since to construct $S^*$ from $S$ we can choose diagonal and 
field automorphisms that normalize every root subgroup of $U$. Of course, the 
BN-pair restricts to $G\cap S^*$. Therefore, if  $G\leq S^*$, we have immediately 
that $N_G(U)$ is maximal in $G$ precisely when $G$ is an extension of some simple 
group that appears in (a).

Suppose then that $G\not\leq S^*$, that is, that $G$ contains a graph or graph-field 
automorphism of $S$. Note that this happens exactly when $S$ is one of the following 
(see \cite{Carter72} or \cite{Atlas}):
$$PSL(n,q), \, n\geq 3,\quad PSp(4,2^f),\quad D_n(q), \quad G_2(3^f),\quad F_4(2^f),
\quad E_6(q).$$
Moreover, non-trivial graph automorphisms, modulo the field automorphisms, 
always have order $2$ or $3$ (order $3$ occurs only in the case $S=D_4(q)$), 
and such automorphisms interchange the fundamental root subgroups. By looking at 
the action of such graph automorphisms on the Dynkin diagrams, only when the Lie rank 
of $S$ is two can it be the case that $N_G(U)$ is maximal. This condition excludes 
all the possible groups except when $S$ is one of the following: $PSL(3,q)$, 
$PSp(4,2^f)$ or $G_2(3^f)$. 
Finally we claim that in these groups $N_G(U)$ is indeed a maximal subgroup of 
$G$. By our earlier argument, the group $G^*=G\cap S^*$ has a BN-pair with Borel 
subgroup $B^*=N_{G^*}(U)$, whose overgroups are $B^*$, $P_1^*$, $P_2^*$ and $G^*$, 
where $P_1^*$ and $P_2^*$ are the meets of $G$ with the extensions, by diagonal and 
field automorphisms, of the two parabolic subgroups of $S$ that contain $B$.
Now $\vert G:G^*\vert = 2$ and any element of $N_G(U)\setminus N_{G^*}(U)$
interchanges $P_1^*$ and $P_2^*$, since it interchanges the two fundamental root 
subgroups. Suppose that $M$ is a maximal subgroup of $G$ containing $N_G(U)$. Then 
$M\cap S^*$ contains $B^* = N_{G^*}(U)$, and so $M\cap S^* \in\{B^*, P_1^*, 
P_2^*,G^*\}$. By the Frattini argument, $G=SN_G(U)$. Thus there is an element 
$g$ in $N_G(U)\setminus N_{G^*}(U)$, such that, by the above, $g$ interchanges 
$P_1^*$ and $P_2^*$. But $g$ normalizes $M\cap S^*$; thus $M\cap S^* = B^*$ or $S^*$. 
If $M\cap S^* = S^*$, then $M = S^*N_G(U) = G$, a contradiction. Hence 
$M\cap S^* = B^*$ and so $M = N_G(U)$. We conclude that $N_G(U)$ is maximal in $G$ in 
all these cases.
\end{proof}

We make a further reduction that applies to any minimal counterexample to 
Theorem \ref{goal}.
\begin{lem}\label{reduction_to_cyclic} 
If $G$ is a minimal counterexample to Theorem \ref{goal}, where $S\leq G\leq \Aut{S}$ 
and $S$ is a finite non abelian simple group, then $G/S$ is a cyclic group.
\end{lem}
\begin{proof}
Assume by contradiction that $G$ has a nilpotent minimal covering with 
$\sigma=\sigma(G)$ subgroups and that $G/S$ is not cyclic. Note that this assumption 
automatically excludes the cases when $S$ is an alternating group $A_n$ with $n\ne 6$ 
or a sporadic group, since in those cases $\vert G/S\vert =2$. Thus $S$ is a simple 
group of Lie type, and, by Lemma \ref{prop_uno} and Proposition \ref{List A}, the 
pair $(G,S)$ is one that appears in the statement of Proposition \ref{List A}. 
We may also assume $S$ is not one of 
$\null^2 B_2(2^f), \null^2 G_2(3^f), PSp(4,2^f),$ 
or $G_2(q)$, since for these groups $\Out{S}$ is cyclic of order $f$ or $2f$. 
Trivially, we may cover $G/S$ using all its non-trivial cyclic subgroups, 
so in particular $\sigma(G/S)< \vert G/S\vert$. 
Since $\sigma\leq \sigma(G/S)$, we deduce that $\sigma<\vert \Out{S}\vert$. 
By Lemma \ref{prop_uno}, then, we have that $n_p(S)<\vert \Out{S}\vert$ 
(where, as before, $n_p(G)$ denotes the number of Sylow $p$-subgroups of $S$, 
that is, the index of a Borel subgroup of $S$ in $S$). 
But for the remaining possible groups listed in Proposition \ref{List A} 
we have
\begin{enumerate}
\item[(a)]$n_p(PSL(2,q)) = q+1$ and $\vert \mbox{Out}(PSL(2,q))\vert = df$, 
where $q = p^f$ and $d = (2,q-1)$, 
\item[(b)]$n_p(PSL(3,q)) = (q+1)(q^2+q+1)$ and 
$\vert \mbox{Out}(PSL(3,q))\vert = 2df$, 
where $q = p^f$ and $d = (3,q-1)$, and
\item[(c)]$n_p(PSU(3,q)) = q^3+1$ and $\vert \mbox{Out}(PSU(3,q))\vert = df$, 
where $q^2 = p^f$ and $d = (3,q+1)$,
\end{enumerate}
and it is straightforward to show in each case that 
$n_p(S)>\vert \mbox{Out}(S)\vert$.
\end{proof}

The following technical lemma is the key ingredient to reduce to the case 
that a minimal counterexample to Theorem \ref{goal} is necessarily a finite 
simple group. 
\begin{lem}\label{centralizers}
Let $G$ be an almost simple group with socle $S$ such that $G/S$ is a cyclic group. 
Assume also that if $S$ is of Lie type, then the pair $(G,S)$ appears in the 
statement of Proposition \ref{List A}. Then there exist some element $s\in S$ 
and some maximal subgroup $K$ of $G$ containing $S$ such that 
$g.c.d.(\vert s\vert, \vert G/K\vert)=1$ and $G\ne KC_G(s)$.
\end{lem}
\begin{proof}
Let $S=A_n$ be an alternating group, with $n\geq 5$.
If $n\ne 6$, or $n=6$ and $G=S_6$, take $s$ to be an $n$-cycle if $n$ is odd,
or an $(n-1)$-cycle if $n$ is even, and take $K=S$.
In both cases $\vert s\vert$ is odd and $C_{S_n}(s)\leq A_n$. 
If $n=6$ and $G$ is a cyclic extension of $A_6$ distinct from $S_6$, 
we may always take $s$ to be a $3$-cycle (see \cite{Atlas}).

If $S$ is a sporadic group, then $\Out{S}$ is always cyclic of order at most two. 
The following table lists possible choices for the order of $s$, depending on the 
pair $(G,S)$ when $G\ne S$ (our reference is \cite{Atlas}). Then $s$ can be chosen to 
be any element of the given order.
\begin{center}
\begin{tabular}{|l|c c c c c c c c c c c c|}
\hline  $S\, $  &  $M_{12}$ & $M_{22} $ & $J_2$ &      $HS$ &  $J_3$  &  $McL$     & 
                $He$     & $Suz$     & $O'N$ & $Fi_{22}$ &  $HN$   &  $Fi_{24}'$ \\ 
\hline  
$\vert s\vert \, $ & $11$   &  $11$     &  $5$  &   $11$   &   $19$  &   $7$      & 
                  $17$   &  $13$     & $31$  &   $13$   &  $19$   &   $29$\\    
\hline 
\end{tabular} 
\end{center}

%\bigskip

We assume now that $G$ is a cyclic extension of a finite simple group $S$ of 
Lie type in characteristic $p$ and that the pair $(G,S)$ satisfies the conclusions of 
Proposition \ref{List A}. 
Let $\delta$ be a diagonal automorphism of $S$ of maximal order $d$, modulo $S$, and
set $\widehat{S}=S\langle \delta\rangle$. Let $\varphi$ be a field automorphism of 
$\widehat{S}$ of order $f$, where $q=p^f$ except when $S$ is unitary, when $q^2=p^f$. 
Set $S^*=\widehat{S}\langle \varphi\rangle$. Then 
$$S\leq \widehat{S} \leq S^*\leq \Aut{S},$$
where the indices are respectively $d$, $f$ and $g$, where $g\in\{1,2\}$, for 
the groups under consideration.
We treat separately the following three cases: a) $G\leq \widehat{S}$, b) $G\leq S^*\setminus 
\widehat{S}$ and c) $G\not\leq S^*$.\\

\noindent
a) Assume $G\leq \widehat{S}$.\\
According to Proposition \ref{List A}, $S\in \{PSL(2,q), \, PSU(3,q)\}$ and 
$G=\widehat{S}$ with the index $d$ of $S$ in $G$ being respectively $2$ or $3$. 
In particular, $p$ is coprime with
$\vert G/S\vert$. Let $s$ be a regular unipotent element of $S$. Since by 
\cite[Lemma 3.1]{White} (respectively by \cite[Table 2]{SimpsonFrame}) we have that 
$C_G(s)<S$, taking $K=S$ we have that $G\ne KC_G(s)$, as we wanted.\\

\noindent
b) Assume $G\leq S^*\setminus \widehat{S}$.\\
According to Proposition \ref{List A}, $S$ is one of the following groups: 
$$PSL(2,q), \, PSU(3,q), \, \null^2 B_2(q),\, \null^2 G_2(q).$$
Note that in the last two cases $q$ is respectively $2^{f}$ or $3^{f}$, with $f$ odd 
and $f\geq 3$ (since $\null^2 B_2(2)$ and $\null^2 G_2(3)$ are not simple groups).
Moreover, as $G\not\leq \widehat{S}$, we always have $f>1$ in this case.
%\noindent
%We may treat the four cases all together. 

Let $\overline{\mathbb{F}}_p$ be the algebraic closure of the 
field $\mathbb{F}_p$ of order $p$. We first claim that for any of the aforementioned 
simple groups $S$  there is a least integer $m$ (whose values are displayed in 
Table A) and an embedding 
$$\iota :\widehat{S}\lto PGL(m, \overline{\mathbb{F}}_p),$$
such that $\varphi$ is the restriction to $S$ of the standard Frobenius automorphism 
of $PGL(m, \overline{\mathbb{F}}_p)$, which later we will still call $\varphi$.
This claim is trivial when $S=PSL(2,q)$ or $S=PSU(3,q)$, respectively, when 
$m=2$ or $3$ 
 and $\iota$ is the natural inclusion. For the case $S=\null^2 B_2(2^{f})$, 
note that $S$ is the centralizer in $S_0=PSp(4,2^{f})$ of a graph involution $x$ 
(\cite[Proposition 2.4.4]{KleidmanLiebeck}) and that 
$$\Aut{S_0}=\Inn{S_0}:(\langle\varphi\rangle \times \langle x\rangle)\simeq S_0:
(C_{f}\times C_2).$$
Since the existence of an embedding $\iota$ with the aforementioned property, 
of $S_0$ into $PGL(4, \overline{\mathbb{F}}_2)$ is guaranteed, the same is true for 
$\null^2 B_2(2^{f})$. Similarly, $S=\null^2 G_2(3^{f})$ is 
the centralizer in $S_0=P\Omega^+(8,3^{f})$ of the full group of graph automorphisms 
of $S_0$
(see \cite{Kleidman}). As
$$\Aut{S_0}=\Inn{S_0}:(\langle\varphi\rangle \times \langle x,y\rangle)\simeq S_0:
(C_{f}\times 
S_3),$$
for suitable graph automorphisms $x$ and $y$, 
and such an embedding $\iota$ exists for $P\Omega^+(8,3^{f})$ into 
$PGL(8, \overline{\mathbb{F}}_3)$, the same is true for $\null^2 G_2(3^{f})$ and our 
claim is proved.

Now we assume that there exists a primitive prime divisor of $p^{fz}-1$, with $z$ as 
in Table A, and let $r$ be such a prime divisor.  Note that $r$ divides the 
order of $S$. Also, trivially, $r\ne d$, and, if $r$ divides $f$, then, writing 
$f=rf'$, we have $0\equiv p^{fz}-1\equiv p^{f'z}-1$ (mod $r$), which contradicts the 
fact that $r$ is a primitive prime divisor of $p^{fz}-1$. Therefore $r$ is coprime 
with $\vert G/S\vert $. Let $t_1$ be an element of $S$ of order $r$.  Note that $t_1$ 
is a power  of a generator of a cyclic maximal torus $T$ of $S$, whose order is 
displayed in Table A.
Suppose that $C_G(t_1)$ contains an element of the form $g\delta^h\varphi^{k}$, 
with $g\in S$, and $0\leq h\leq d-1$, $0< k \leq f-1$. Then 
$$t_1^{g\delta^h}=t_1^{\varphi^{-k}}.$$
This, of course, implies that
$$(\iota (t_1^{g\delta^h}))^L=(\iota (t_1)^{\varphi^{-k}})^L,$$
where $L=PGL(m,\overline{\mathbb{F}}_p)$ and $(y)^L$ denotes the $L$-conjugacy class 
of $y\in L$. Now $\iota (t_1)$ is $L$-conjugate to the projection $\bar\alpha$ of a 
diagonal $m\times m$ matrix $\alpha$, and $\varphi$ sends $\bar\alpha$ to its 
$p$-th power $\bar\alpha^p$.
As $\iota(g\delta^h)\in L$, it follows that
$$(\iota (t_1^{g\delta^h}))^L=(\iota (t_1))^L=(\bar\alpha)^L=(\bar\alpha^{p^{-k}})^L.
$$
We want to prove that if $0<k<f$ the two $L$-classes $(\bar\alpha)^L$ and 
$(\bar\alpha^{p^{-k}})^L$ are different.
Note that as $t_1$ has order $r$ and the matrix $\alpha\in SL(m,
\overline{\mathbb{F}}_p)$ is determined modulo the scalars, we can choose $\alpha$ in 
such a way that its eigenvalues are either $1$ or  have order $r$ in the 
multiplicative group of $\overline{\mathbb{F}}_p$. Moreover, $\alpha$ has at least 
one eigenvalue $\mu$ of order $r$. Note that $\mu$ belongs to the field 
$\mathbb{F}_{p^{fz}}$, but to no smaller field. Since $\iota (t_1)\in PSL(m,p^f)$, 
the characteristic polynomial $\chi$ of $\alpha$ has coefficients in 
$\mathbb{F}_{p^f}$, so $\mu, \mu^{p^f}, \dots,\mu^{p^{f(z-1)}}$ are all distinct 
roots of $\chi$. If $S\ne \null^2G_2(q)$ then $\chi$ has degree $m=z$ and the 
eigenvalues of $\alpha$ are precisely $\mu, \mu^{p^f}, \dots, \mu^{p^{f(z-1)}}$. 
If $S=\null^2G_2(q)$, then $m=6$, $z=8$, and $\chi$ factors as $\chi=\chi_1\chi_2$, 
where $\chi_1$ is the minimum polynomial of $\mu$ and has degree $6$, and $\chi_2$ 
has degree $2$. Now the roots of $\chi_2$, which are eigenvalues of $\alpha$, cannot 
have order $r$, because $r\nmid q^2-1$, so they must be $1$, and the eigenvalues of 
$\alpha$ are: $\mu, \mu^q,\dots,\mu^{q^{5}}, 1, 1$.
The non-zero entries of the matrix $\alpha^{p^{-k}}$ are the $p^{-k}$-th powers of 
the eigenvalues of $\alpha$ and it is straightforward to see that 
if $0<k<f$  it cannot happen that $\alpha=\lambda\alpha^{p^{-k}}$ for some 
$\lambda\in \overline{\mathbb{F}}_p$. This proves that $C_G(t_1)\leq G\cap 
\widehat{S}$. Thus, by taking $s=t_1$ and $K$ any maximal subgroup containing 
$\widehat{S}\cap G$, we have that $G\ne KC_G(s)$, as we wanted.

We consider now the cases in which no primitive prime divisor of 
$p^{fz}-1$ exists.
Then by Zsigmondy's Theorem (see \cite{Zsig1892}), either $(p,zf)=(2,6)$ 
or $p$ is a Mersenne prime and $zf=2$. The last condition 
cannot happen, since  in this case both $z$ and $f$ are greater than one. 
The first condition reduces to considering the 
cases when $S$ is either $PSL(2,8)$ or $PSU(3,2)$. But $PSU(3,2)$
is not simple, while if $S=PSL(2,8)$, then $G=\Aut{S}=S^*$, and 
we can take $s$ to be an element of $S$ of order $7$ and $K=S$ (see \cite{Atlas}). 

\begin{center}
Table A\\
\begin{tabular}{|l|l|l|l|l|l|}
\hline $S$  & $\vert T\vert $, $T$ a max. torus of $S$ & $d$ & 
$\vert \Out{S}\vert $ & $z$ & $m$\\ 
\hline  
$PSL(2,q)$     & $(q+1)/d $      & $(q-1,2)$ & $df$ & 2 & 2   \\
$PSU(3,q)$     & $(q^2-q+1)/d$   & $(q+1,3)$ & $df$ & 3 & 3 \\
$\null^2 B_2(q)$   & $\left\{ \begin{array}{ll}
q+\sqrt{2q}+1 & \textrm{or}\\
q-\sqrt{2q}+1 & \textrm{} 
\end{array} \right. $  & 1 & $f$ & 4 & 4\\
$\null^2 G_2(q)$   & $\left\{ \begin{array}{ll}
q+\sqrt{3q}+1 & \textrm{or}\\
q-\sqrt{3q}+1 & \textrm{} 
\end{array} \right. $  & 1 & $f$ & 6 & 8\\
\hline 
\end{tabular} 
\end{center}

\noindent 
c) Assume $G\not \leq S^*$. \\
Then, according to Proposition \ref{List A}, $S$ is one of the following groups: 
$$PSL(3,q),\, PSp(4,2^f),\, G_2(3^f),$$ 
with $f$ an integer greater than 1, and $\vert G:G^*\vert =2$, where $G^*=G\cap S^*$.
We choose $s$ to be a generator of  a cyclic maximal torus $T$ of $S$, whose order is 
respectively $(q^2+q+1)/d$, $q^2+1$, or $q^2-q+1$, according to whether $S$ is 
$PSL(3,q)$, $PSp(4,2^f)$ or $G_2(3^f)$, and $K=G^*$. 
Note that $\vert s\vert = \vert T\vert $ is odd, and thus coprime with 
$\vert G/K\vert $. 
We first claim that $\vert C_{G^*}(T)\vert $ is odd. 
If not, let $y$ be an involution in $C_{G^*}(T)$. 
Since $\vert C_{\widehat{S}}(T)\vert = d\vert T\vert$ is odd, $y\not\in \widehat{S}$. 
By Proposition 4.9.1 in \cite{GLS3}, we have that $f$ is even and $y$ is 
$\widehat{S}$-conjugate to a field automorphism of order two. 
In particular, $C_S(y)$ is isomorphic respectively to 
$PSL(3,p^{f/2})$, $PSp(4,2^{f/2})$ or $G_2(3^{f/2})$. 
In each of these cases, by order reasons, $C_S(y)$ cannot contain $T$. Thus 
$\vert C_{G^*}(T)\vert $ is odd, and if we argue by contradiction assuming 
$G=G^*C_G(s)$, 
there exists some involution $x$ in $C_G(s)\setminus G^*$. 
Again by Proposition 4.9.1 in \cite{GLS3}, we have that $C_G(x)$ is isomorphic 
respectively to $PSU(3,q)$, $\null^2 B_2(q)$ or $\null^2 G_2(q)$. 
But none of these groups contains a cyclic maximal torus $T$ of $S$, a contradiction.
\end{proof}

The following proposition eliminates the possibility that a minimal counterexample to 
Theorem \ref{goal} can be an almost simple group but not simple.
\begin{pro}\label{reduction_to_simple}
Let $G$ be an almost simple group which is not simple. Then $G$ does not admit a 
nilpotent minimal covering.
\end{pro}
\begin{proof}
Suppose that $G$ has a nilpotent minimal covering 
$$G=A_1\cup\ldots\cup A_{\sigma},$$
with $\sigma =\sigma(G)$, and all $A_i$ nilpotent. Suppose further that 
$S<G\leq \Aut{S}$, where $S$ is a finite non abelian simple group. 
If $S$ is of Lie type, then by Lemma 
\ref{prop_uno} the pair $(G,S)$ is one that appears in Proposition 
\ref{List A}. 
Also by Lemma \ref{reduction_to_cyclic}, we can assume that $G/S$ is cyclic. 
According to Lemma \ref{centralizers}, we may choose an element $s$ in $S$ and a 
maximal subgroup $K$ of $G$ containing $S$ such that $\vert s\vert $ is coprime with 
the prime $r=\vert G/K\vert$ and $G\ne KC_G(s)$. 
Note that $r$ is prime since $G/S$ is cyclic.
Let $s\in A_i$ for some $i \in \{1, \ldots, \sigma\}$. We claim that $A_i$ lies in $K$. 
For if not let $\alpha\in A_i\setminus K$, so that $G=K\langle \alpha\rangle$ by the 
maximality of $K$. If $\vert\alpha\vert=r^av$, with $r$ not dividing $v$, we also 
have that $G=K\langle\alpha^v\rangle$, and, moreover, that $\alpha^v$ is an element 
of $A_i$ of order $r^a$, which is coprime with $\vert s\vert$. Thus, since $A_i$ is 
nilpotent, $\alpha^v\in C_G(s)$, forcing $G$ to be equal to $KC_G(s)$, contradicting 
the choices of $s$ and $K$. 
%
% Note that such an element always exists. If $\alpha\in A_1$ and $G/S$ is generated by $S\alpha$, we write $\vert \alpha \vert =ab$ with $a$ coprime with $\vert s\vert $ and $\pi(b)\subseteq \pi(\vert s\vert)$. %Then $\alpha=\alpha^{ax}\alpha^{by}$, for some $x,y\in \mathbb{Z}$, and note that $\alpha^{by}\not\in M$, for %every $M$ maximal in $G$, $S\leq M$, otherwise $M\alpha=M\alpha^{ax}$ has order coprime with $G/S$, forcing $M\alpha=M$, contradiction. Therefore $\alpha^{by}\in A_1\setminus \bigcup_{S\leq M<G}M$ and has order dividing $a$, so coprime with $\vert s\vert$. Being $A_1$ nilpotent, $\alpha$ centralizes $s$, forcing $G=SC_G(s)$, which contradicts
% 
Thus $A_i\leq K$. We may choose some $g\in S$ such that $s^g\not\in A_i$.
Such a $g\in S$ exists, for otherwise by the simplicity of $S$ we would have that 
$\langle s^g\vert g\in S\rangle=S\leq A_i$, contradicting the nilpotence of $A_i$. 
Suppose that $s^g\in A_j$. Then, arguing as before, $A_j\leq K$, and so we conclude 
that $G=\langle A_i,A_j\rangle\leq K$, a contradiction. 
\end{proof}
%
%\noindent
%{\bf{Case: $G$ simple, $G=S$.}}\\
%
We are now in a position to complete the proof of Theorem \ref{goal}. 
\begin{pro}\label{end_simple}
No finite simple group $S$ admits a nilpotent minimal covering.
\end{pro}
\begin{proof}
In \cite{BS2008} the cases $S$ alternating and sporadic are completely settled.  
We can assume therefore that $S$ is a finite simple group of Lie type in 
characteristic $p$. By Lemma \ref{prop_uno} and Proposition \ref{List A}, 
$S$ lies in one of the following families:
$$ PSL(2,q), PSU(3,q), \null^2 B_2(q), \null^2 G_2(q).$$
The two families of projective special linear groups $PSL(2,q)$ and the Suzuki groups
$\null^2 B_2(q)$ have also been settled 
in \cite{BS2008} 
(respectively Lemma 4.2 and Theorem 4.3). We need only to analyze the 
two remaining families. 

Let therefore $S=PSU(3,q)$, with $q>2$, or $S=\null^2 G_2(q)$, with $q=3^f$, where 
$f$ is odd and $f\geq 3$. Note that 
$\vert PSU(3,q)\vert = \frac{1}{d}\cdot q^3(q^2-1)(q^3+1)$, where $d=(3,q+1)$, and 
$\vert \null^2 G_2(q)\vert =q^3(q-1)(q^3+1)$. Assume first that there is an odd prime 
$r$ dividing $q-1$ (observe that this is always the case when $S=\null^2 G_2(q)$). 
Since $q\equiv 1$ (mod $r$), we have that $r$ is coprime with $\vert S\vert/(q-1)$. 
Thus if $R$ is a Sylow $r$-subgroup of $S$, $R$ lies in a Levi complement $H$ of a 
Borel subgroup $B=UH$. In particular, $R$ is cyclic, $R=\langle x\rangle$, 
and $H = C_S(R)$, by \cite[Table 2]{SimpsonFrame} for $PSU(3,q)$, and 
\cite{Ward-Ree}, or \cite[Lemma 2.2]{Lucido2007} for $\null^2 G_2(q)$. If $x\in 
A_i$, then, since $A_i$ is nilpotent and the Sylow $r$-subgroups are cyclic, 
$A_i\le C_S(x) = H$.  Now, let $u\in U$ be a regular unipotent element of $S$. Then 
$C_S(u)$ is a $p$-subgroup (again by \cite[Table 2]{SimpsonFrame} and 
\cite{Ward-Ree}). In particular if $u\in A_j$, then $A_j\le U$.  But then we get a 
contradiction, since $\langle A_i,A_j \rangle \leq B$.

It remains to consider the case when $S=PSU(3,q)$ and $q-1$ is a power of $2$. 
Note that this happens if and only if $q=9$ or $q$ is a Fermat prime, say $q=2^m+1$. 
Let $\mathcal{A}=\{A_i\}_{i=1}^{\sigma}$ be a nilpotent minimal covering of \
$S=PSU(3,q)$. 
The centralizer of any regular unipotent element $u$ of $S$ is a $p$-subgroup 
(\cite[Table 2]{SimpsonFrame}), and therefore there exists a unique maximal 
nilpotent subgroup of $S$ containing $u$, and this subgroup is a Sylow $p$-subgroup 
of $S$. We may therefore assume that all the Sylow $p$-subgroups of $S$ appear 
as members of the nilpotent covering $\mathcal{A}$. Now let $U$ be a Sylow 
$p$-subgroup and 
let $H$ be a Levi complement of it in a Borel subgroup $N_S(U)=UH$. In 
particular, $H$ is cyclic of order $(q^2-1)/d$. Let $h$ be a generating element of 
$H$ and assume that $h\in A_{i}$ for some $i\in\{1,\ldots ,\sigma\}$. 
If $A_i=H$ then we can replace the subgroups $U$ and $A_i$ of $\mathcal{A}$ 
with the subgroup $N_G(U)$, obtaining a covering of $S$ with fewer than $\sigma$ 
members, which contradicts the minimality of $\sigma$. 
Therefore $A_i$ must be a nilpotent subgroup of $S$ that strictly contains $H$. 
Now the Sylow $2$-subgroup of $H$ is cyclic (of order $16$ if $q=9$ and $2^{m+1}$ if 
$q=2^m+1$) and so it is normal of index $2$ in a Sylow $2$-subgroup of $S$. In 
particular, the involution $w$ of $H$ is a central element of $A_i$, that is, 
$A_i\leq C_S(w)$.  By Proposition~4~(iii) in \cite[Chapter II, Section 2]
{AlperinBrauerGorenstein}, $C_S(w)$ is a central extension of a 
cyclic group of order $\frac{q+1}{d}$ by a group isomorphic to $PGL(2,q)$. 
%\quad [\simeq \frac{GU(2,q)}{3}].$$
Note that $H^{q-1}$ is the central subgroup of $C_S(w)$ of order $(q+1)/d$. 
Now the only nilpotent subgroups of $C_S(w)$ that strictly contain a cyclic subgroup 
of order $(q^2-1)/d$ are central extensions of $C_{\frac{q+1}{d}}$ by a Sylow 
$2$-subgroup, say $P$, of $PGL(2,q)$ (and so also of $S$). 
Therefore $A_i$ is a group isomorphic to $C_{\frac{q+1}{2d}}\times P$, 
and it contains $H$ as a 
subgroup of index two. Since $H$ is not normal in $C_S(w)$ we can find an element 
$g\in C_S(w)\setminus N_S(H)$ and consider the element $h^g$. Assume that 
$h^g\in A_j$. 
Arguing as before, we have that either $A_j=\langle h^g\rangle=H^g$, or $A_j$ is a 
subgroup of $C_S(w)^g=C_S(w)$ isomorphic to $A_i$. In the first case we obtain a 
contradiction by replacing the subgroups $U^g$ and $A_j$ in $\mathcal{A}$ with 
$(UH)^g$. In the latter case we have that $A_j\ne A_i$, since $H\ne H^g$ and a 
group isomorphic to $A_i$ has a unique cyclic maximal subgroup of index two. But 
then $G=\langle A_i, A_j \rangle \leq C_S(w)$, a contradiction.
\end{proof}

\section{Acknowledgments} 
We are very grateful to Ron Solomon for his precious help 
in dealing with the groups of Lie type and for the constant encouragement he gave us. 
The first-listed author is thankful for the generous hospitality of the Dipartimento 
di Matematica e Informatica of the Universit\`a di Firenze during his visit there, 
when this work was undertaken. The third-listed author is partially supported by the 
``National Group for Algebraic and Geometric Structures, and their Applications'' 
(GNSAGA - INDAM).

\end{document}